\documentclass[11pt]{article}
\usepackage{amsmath}
\usepackage{amsfonts}
\usepackage{graphicx}
\usepackage{amssymb}
\usepackage{leftidx}
\usepackage{color}
\usepackage{bbm}
\usepackage{mathrsfs}

\allowdisplaybreaks

\newtheorem{condition**}{A*}
\newtheorem{condition***}{C*}
\newtheorem{condition*}{C}

\newtheorem{definition}{Definition}[section]
\newtheorem{theorem}{Theorem}[section]
\newtheorem{lemma}{Lemma}[section]
\newtheorem{remark}{Remark}[section]

\begin{document}

\title{Optimal Transport for Continuous-Time Mean-Variance Portfolio Selection} 

\author{Yijie Zeng$^{2}$, \quad 
Liangquan Zhang$^{1}$\thanks{The corresponding author. 
Liangquan Zhang acknowledges the financial support partly by the National Nature
Science Foundation of China (Grant No. 12571486, 12171053) and the
Fundamental Research Funds for the Central Universities, and the Research
Funds of Renmin University of China (No. 23XNKJ05).}, \quad
Xun Li$^{2}$\thanks{
Xun Li is supported by the Research Grants Council of Hong Kong under grants
15225124 and 15237226. E-mail: li.xun@polyu.edu.hk.} \\
{\small 1. School of Mathematics } \\
{\small Renmin University of China, Beijing 100872, China.} \\
{\small 2. Department of Applied Mathematics } \\
{\small \ The Hong Kong Polytechnic University, Hong Kong, China.}}

\maketitle

\begin{abstract}

The continuous-time mean-variance portfolio selection
problem is investigated in this paper through a rigorous integration of
backward stochastic differential equation (BSDE) and Kantorovich duality
theory. We propose a novel analytical framework for addressing feedback
control in McKean-Vlasov stochastic systems. By solving three coupled
systems of ordinary differential equations (ODEs), we derive an explicit,
closed-loop optimal feedback investment strategy.
\end{abstract}



\section{Introduction}\label{sect:1}

\paragraph{Markowitz Mean-Variance (MV) Model.}

The foundational work on portfolio selection theory was introduced by Harry Markowitz in 1952 (cf. \cite{m52, m59}), in which he formally proposed the ``Mean-Variance Model" a milestone that inaugurated modern investment theory. The core objective of this model is to minimize
investment risk under a given expected return level, or to maximize expected
return under a given risk level, by quantifying the ``return-risk" trade-off
relationship, providing a scientific basis for multi-asset portfolio
allocation. Return: Measured by the expected return (mean) of the asset
portfolio, reflecting the average level of future returns of the portfolio;
Risk: Measured by the variance (or standard deviation) of the asset
portfolio's return, reflecting the degree of fluctuation (the greater the
fluctuation, the higher the risk). The core logic of diversification
investment: Portfolio risk not only depends on the risk of individual assets
(variance), but also on the correlation (covariance) between assets. By
allocating assets with lower correlation, the overall risk of the portfolio
(i.e., ``non-systematic risk diversification") can be reduced without
lowering the expected return.

The derivation of the model is based on the following simplified
assumptions, which are also the root cause of its subsequent limitations:
Asset returns follow a normal distribution (only the first two moments- mean
and variance/covariance are sufficient to fully characterize the
distribution); Investors are risk-averse: Preferring combinations with lower
risk for the same return, and preferring combinations with higher return for
the same risk;

Investment decisions are single-period static optimization: Considering only
one investment period, without dynamic adjustment and transaction costs; No
frictional market: No transaction fees, taxes, liquidity costs, and assets
can be infinitely subdivided (supporting any proportion allocation);
Investors only focus on return (mean) and risk (variance), ignoring
higher-order moments (skewness, kurtosis) and other non-quantifiable factors.

Let us briefly introduce the basic mathematical modeling framework:

\begin{itemize}
\item Asset set: $n$ types of assets, weight vector $\boldsymbol{w}=(w_{1},w_{2},...,w_{n})^{T}$ (where $w_{i}\geq 0$ is the allocation
proportion of the $i$-th asset, $\sum_{i=1}^{n}w_{i}=1$;

\item Individual asset: The expected return of the $i$-th asset $\mu _{i}=\mathbb{E}[r_{i}]$, variance $\sigma _{i}^{2}=\text{Var}(r_{i})$;

\item Asset correlation: The covariance between the $i$-th and $j$-th assets
$\sigma _{ij}=\text{Cov}(r_{i},r_{j})$, covariance matrix $\boldsymbol{\Sigma}=(\sigma _{ij})_{n\times n}$;

\item Portfolio's Return and Risk Portfolio expected return (mean): $\mu
_{p}=\boldsymbol{w}^{T}\boldsymbol{\mu }=\sum_{i=1}^{n}w_{i}\mu _{i}$
(weighted sum of the expected returns of individual assets); Combined
variance (risk): $\sigma _{p}^{2}=\boldsymbol{w}^{T}\boldsymbol{\Sigma}\boldsymbol{w}=\sum_{i=1}^{n}w_{i}^{2}\sigma _{i}^{2}+2\sum_{1\leq i<j\leq
n}w_{i}w_{j}\sigma _{ij}$ (including the variance term of individual assets
and the covariance term between assets).
\end{itemize}

The essence of the optimization objective and the constraint model is a
quadratic convex programming problem, and there exists a unique global
optimal solution.

\begin{itemize}
\item Objective 1: Given the expected return $\mu ^{\ast }$, minimize the risk
\begin{equation*}
\left\{
\begin{array}{l}
\min \sigma _{p}^{2}=\boldsymbol{w}^{T}\boldsymbol{\Sigma w}, \\
\text{s.t.}\quad \boldsymbol{w}^{T}\boldsymbol{\mu }=\mu ^{\ast }, \\
\sum_{i=1}^{n}w_{i}=1, \\
w_{i}\geq 0\quad (w_{i}\in \mathbb{R},\text{ short selling is allowed}). 
\end{array}\right.
\end{equation*}

\item Objective 2: Given the risk $\sigma ^{\ast }$, maximize the expected
return
\end{itemize}

\begin{equation*}
\left\{
\begin{array}{l}
\max \mu _{p}=\boldsymbol{w}^{T}\boldsymbol{\mu }, \\
\text{s.t.}\quad \boldsymbol{w}^{T}\boldsymbol{\Sigma w}=\sigma ^{\ast 2}, \\
\sum_{i=1}^{n}w_{i}=1, \\
w_{i}\geq 0\quad (w_{i}\in \mathbb{R},\text{ short selling is allowed}). 
\end{array}%
\right.
\end{equation*}
\ \quad \quad

Solving the above optimization problem yields a series of \textquotedblleft
risk-return\textquotedblright\ optimal portfolios. The curve formed by these
portfolios is called the \textquotedblleft efficient frontier\textquotedblright. The portfolios on the
efficient frontier are \textquotedblleft non-dominated\textquotedblright\ --- there is no other portfolio that
\textquotedblleft has lower risk while having higher return\textquotedblright, or \textquotedblleft has higher return while
having lower risk\textquotedblright; portfolios below the efficient frontier are \textquotedblleft inefficient\textquotedblright\ (they can be optimized by adjusting weights).

\paragraph{Continuous MV Model.}

Consider a finance market with $d+1$ assets which evolve continuously on $[0,T]$. One asset is \emph{riskless}, for instance, bond, whose price
denoted by $S_{0}(t),\;t\geq 0$, is driven by ODE
\begin{equation}
\left\{\begin{aligned} \mathrm{d} S_0(t) 
& = r(t) S_0(t) \mathrm{d} t, \quad t \in[0, T], \\ S_0(0) & =s_0>0, \end{aligned}\right.  \label{b}
\end{equation} 
where $r(t)>0,$ denote the interest rate of the bond.

The other assets are \emph{risky}, for instance, stocks, and their prices
are governed by
\begin{equation}
\left\{
\begin{array}{lll}
\mathrm{d}S_{i}\left( t\right) 
& = & S_{i}\left( t\right) \left\{b_{i}\left( t\right) \mathrm{d}t+\sum_{j=1}^{m}\sigma _{ij}\left( t\right)\mathrm{d}W^{j}\left( t\right) \right\} ,\text{ }t\in \left[ 0,T\right], \\
S_{i}(0) & = & s_{i} > 0,
\end{array}\right.  \label{s}
\end{equation}
where $b_{i}(t)>r(t),$ $t\in \lbrack 0,T]$ denotes the appreciation rate and
$\sigma _{ij}(t)$ is the volatility term. Put $b:=\left( b_{1},\cdots,b_{d}\right) ^{\top}$ and $\sigma _{i}(t):=\left( \sigma _{i1}(t),\ldots,\sigma _{im}(t)\right)$.

Define the volatility matrix
\begin{equation*}
\sigma (t) = \left(\begin{array}{c}
\sigma _{1}(t) \\
\vdots \\
\sigma _{d}(t)
\end{array}\right) 
= \left( \sigma _{ij}(t)\right) _{m\times m}\text{ }.
\end{equation*} 
We assume throughout that $r,$ $b,$ and $\sigma $ are deterministic
constants. In addition, we impose that the non-degeneracy condition $\sigma
(t)\sigma ^{\top }(t)\geq \delta I,$ $\forall t\in \lbrack 0,T],$ where $\delta >0$ is a given constant, is satisfied. 

Consider an investor with an initial wealth $x_{0}>0$ and the her or his
total wealth $X(t)$ (see, e.g., \cite{lz2000}) is driven by
\begin{equation}
\left\{ \begin{aligned} \mathrm{d} x(t) 
= & \left\{r(t)x(t)+\sum_{i=1}^m\left(b_i(t)-r(t)\right) u_i(t)\right\} \mathrm{d} t \\ 
& +\sum_{j=1}^m \sum_{i=1}^m \sigma_{i j}(t) u_i(t) \mathrm{d} W^j(t), t\in[0, T], \\ 
x(0)= & x_0 , \end{aligned}\right.  \label{w}
\end{equation}
where $\left\{ u_{i}(t)\right\} _{i=0,1,\cdots ,m}$ stands for the total
market value of the agent's wealth in the $i$-th bond/stock. In the
literature $u(t):=(u_{1}(t),\cdots ,u_{m}(t))$ is called the portfolio.

According to mean-variance portfolio selection theory, the risk measure is defined by
\begin{equation*}
J(u)=\frac{\lambda }{2}\mathbb{D}[X(T)]-\mathbb{E}[X(T)],
\end{equation*}
where $\lambda >0$ and $\mathbb{D}[X(T)]$ denotes the variance of random
variable $X(T)$. The aim is to seek a permissible investment strategy (a
dynamic portfolio fulfilling all constraints) such that $J(u)$ is minimized.

\begin{definition}
\label{defi:control-positive} We say a portfolio $u(\cdot )$ is admissible
whenever $u(\cdot )\in L_{\mathcal{F}}^{2}(0,T;\mathbb{R}^{m})$.
\end{definition}

The mean--variance portfolio selection problem is formulated as the
following optimization problem
\begin{equation}
\begin{array}{rl}
\min & J\left( u\right) , \\
\mbox{{\rm subject to}} & \left\{
\begin{array}{l}
u(\cdot )\in L_{\mathcal{F}}^{2}(0,T;\mathbb{R}^{m}), \\
(X(\cdot ),u(\cdot ))~\hbox{\rm admit (\ref{w})}.%
\end{array}%
\right.%
\end{array}
\label{eq:utility-positive}
\end{equation} 
Moreover, the optimal control of (\ref{eq:utility-positive}) denoted by $u^{\ast }$ is called an efficient strategy, $J\left( u^{\ast }\right) $ is
the optimal value of (\ref{eq:utility-positive}) corresponding to $u^{\ast }$. 

The contributions of this paper can be summarized as follows: By quantifying
the overall differences in the distribution rather than merely focusing on
the first two moments, it is possible to more accurately control the extreme
risks in non-normal markets.

The cost functional is
\begin{equation*}
\mathcal{J}\left( u\right) =\mathcal{W}_{2}\left( \mu _{T},\mu _{A}\right)^{2}.
\end{equation*} 
The main academic limitation of the Markowitz MV model lies in its reliance
on the strong assumption that \textquotedblleft asset returns follow a normal distributio\textquotedblright\ ---
under this assumption, the distribution can be fully characterized by the
mean (first moment) and variance (second moment), simplifying the
optimization logic. However, this assumption has long been disproved in
academic empirical studies (as financial asset returns typically exhibit
non-normal characteristics such as sharp peaks, heavy tails, and skewness).

The academic breakthrough of the Wasserstein distance lies in:

\begin{itemize}
\item \emph{Full distribution characterization without distribution
assumptions}: The definition of Wasserstein distance does not rely on any
specific distribution form. It quantifies distribution differences through
\textquotedblleft coupling of probability measures\textquotedblright ,
providing an academic proof that \textquotedblleft portfolio optimization
can be independent of the normal distribution assumption\textquotedblright ,
upgrading the theoretical basis of the MV model from \textquotedblleft local
moment optimization under specific distributions\textquotedblright\ to
\textquotedblleft full distribution optimization under arbitrary
distributions\textquotedblright .

\item \emph{Solving the academic dilemma of} \textquotedblleft \emph{missing
moment information}\textquotedblright : In non-normal distributions,
higher-order moments (skewness, kurtosis) have significant impacts on
risk-reward characteristics, but traditional MV models cannot incorporate
these information (otherwise, it would lead to \textquotedblleft moment
explosion\textquotedblright\ and the optimization problem would be
unsolvable). The Wasserstein distance naturally integrates all distribution
information (including higher-order moments and tail characteristics)
through \textquotedblleft minimum transportation cost\textquotedblright ,
constructing a theoretical path \textquotedblleft without separately
modeling higher-order moments, yet covering all distribution
characteristics\textquotedblright , avoiding the mathematical complexity of
higher-order moment optimization.
\end{itemize}

Academic significance: It corrects the core premise defect of the MV model,
making the assumptions of portfolio selection theory more in line with the
empirical characteristics of the financial market, and promoting the
evolution of the MV model from \textquotedblleft theoretical model under ideal conditions\textquotedblright\ to \textquotedblleft rigorous framework under real distribution\textquotedblright.

\paragraph{Wasserstein distance in deep learning.}

The control of probability distributions holds significant potential across
a range of applications, including autonomous vehicle navigation \cite{ot19}
and the management of car-sharing systems \cite{hs21}. Moreover, this
concept is expected to play a valuable role in advancing deep learning
methodologies. A central challenge in deep learning involves the effective
modeling and manipulation of data and their underlying distributions. In
particular, generative modeling---exemplified by approaches such as
generative adversarial networks \cite{gpmx14, acb17}, normalizing flows \cite{rm15}, and diffusion models \cite{sskk20}---aims to construct probability
distributions that closely approximate empirical data distributions.
Consequently, the design of neural networks can be interpreted as an
optimization process over the probability distributions induced by the
network architecture. Recent studies have explored generative models from
the standpoint of stochastic control theory \cite{zc21, lct22, bru22},
further establishing connections between distributional control and learning
dynamics.

In deep learning, the Wasserstein distance is commonly employed as a loss
function and is minimized during the training process \cite{acb17}.
Furthermore, generative modeling networks can be formulated as dynamical
systems, as illustrated in the study of neural ordinary differential
equations \cite{cybd18}. In this framework, determining the network weights
can be interpreted as identifying an appropriate control for a dynamical
system. Consequently, minimizing the Wasserstein distance in deep learning
can be conceptualized as an optimal control problem governed by the
Wasserstein metric. This formulation of optimal control with respect to the
Wasserstein distance is expected to provide a theoretical foundation for
generative modeling.

\paragraph{Feedback control.}

Feedback Control is the core paradigm of a closed-loop control system. Its
core logic involves a cyclic mechanism of real-time measurement of system
output--comparison with the target value--dynamic adjustment of input based
on the deviation, which is used to correct system disturbances and track the
preset target (cf. Wiener, 1948, \cite{wiener}, Chapter IV, Feedback and
Oscillation). It is the fundamental technical framework in modern control
theory, engineering systems, financial management, biological regulation,
etc. Unlike "open-loop control" (which only follows preset instructions
without output correction), the core value of feedback control lies in
anti-interference, stable dynamics, and adaptability, which is highly
consistent with what you are concerned about, such as optimal control,
dynamic combination optimization, and robust design.

Feedback control provides a closed-loop dynamic adjustment methodology for
dynamic MV models and distributed robust combined optimization, addressing
the core deficiencies of traditional static models. It is an important
technical support on the extended Wasserstein distance MV model --- by
real-time correcting distribution deviations through feedback control, it
achieves the dual goals of full distribution robustness and dynamic
stabilization control.

The remainder of the paper is organized as follows. Section \ref{sect2}
presents preliminary results, reformulating the problem by involving
Kantorovich duality theorem. In Section \ref{sect3}, we derive the optimal
feedback control strategy and groups of related ODEs using the maximum
principle by Lions derivative. Concluding remarks are provided in Section \ref{sect4}.

\section{Preliminaries}

\label{sect2}

\subsection{Notation}

We make use of the following notation:

\begin{tabular}{crl}
$M^{\top }$ & : & the transpose of any matrix or vector $M$; \\
$\Vert M\Vert $ & : & $\sqrt{\sum_{i,j}m_{ij}^{2}}$ for any matrix or vector
$M=(m_{ij})$; \\
$\mathbb{R}^{n}$ & : & $n$ dimensional real Euclidean space; \\
$\mathbb{R}_{+}^{n}$ & : & the subset of $\mathbb{R}^{n}$ consisting of
elements with nonnegative components.%
\end{tabular}

The underlying uncertainty is generated by a fixed filtered complete
probability space $(\Omega ,\mathcal{F},\mathbb{P},\{\mathcal{F}%
_{t}\}_{t\geq 0})$ on which is defined a standard $\{\mathcal{F}%
_{t}\}_{t\geq 0}$-adapted $m$-dimensional Brownian motion $W(t)\equiv
(W^{1}(t),\cdots ,W^{m}(t))^{\top }$. $\mathbb{F}=\left( \mathcal{F}%
_{t}\right) _{0\leq t\leq T}$ is the natural filtration generated by $W.$
Given a probability space $(\Omega ,\mathcal{F},\mathbb{P})$ with a
filtration $\{\mathcal{F}_{t}|a\leq t\leq b\},$ $(-\infty \leq a<b\leq
+\infty )$, a Hilbert space $\mathcal{H}$ with the norm $\Vert \cdot \Vert _{%
\mathcal{H}}$, define the Banach space
\begin{equation*}
L_{\mathcal{F}}^{2}(0,T;\mathcal{H})=\left\{ \varphi (\cdot )\bigg\vert%
\begin{array}{l}
\varphi (\cdot )\mbox{ is an }\mathcal{F}_{t}\mbox{-adapted, }\mathcal{H}%
\mbox{-valued measurable process on }\lbrack a,b] \\
\mbox{and }\mathbb{E}\left[ \int_{a}^{b}\Vert \varphi (t)\Vert _{\mathcal{H}%
}^{2}dt\right] <+\infty%
\end{array}%
\right\}
\end{equation*}%
with the norm
\begin{equation*}
\Vert \varphi (\cdot )\Vert _{\mathcal{F},2}=\Big(\mathbb{E}\left[
\int_{a}^{b}\Vert \varphi (t)\Vert _{\mathcal{H}}^{2}dt\right] \Big)^{\frac{1%
}{2}}<+\infty .
\end{equation*}%

Let $(\widetilde{\Omega },\widetilde{\mathcal{F}},\widetilde{P})$ be a copy
of the probability space $(\Omega ,\mathcal{F},P)$. For any random variable
(of arbitrary dimension) $\vartheta $ over $(\Omega ,\mathcal{F},P)$ we
denote by $\widetilde{\vartheta }$ a copy (of the same law as $\vartheta $,
but defined over $(\widetilde{\Omega },\widetilde{\mathcal{F}},\widetilde{P}%
) $: $\widetilde{P}_{\widetilde{\vartheta }}=P_{\vartheta }$. The
expectation $\widetilde{\mathbb{E}}[\cdot ]=\int_{\widetilde{\Omega }}(\cdot
)d\widetilde{P}$ acts only over the variables endowed with a tilde. This can
be made rigorous by working with the product space $(\Omega ,\mathcal{F}%
,P)\otimes (\widetilde{\Omega },\widetilde{\mathcal{F}},\widetilde{P}%
)=(\Omega ,\mathcal{F},P)\otimes (\Omega ,\mathcal{F},P)$ and putting $%
\widetilde{\vartheta }(\widetilde{\omega },\omega ):=\vartheta (\widetilde{%
\omega })$, $(\widetilde{\omega },\omega )\in \widetilde{\Omega }\times
\Omega =\Omega \times \Omega $, for $\vartheta $ random variable defined
over $(\Omega ,\mathcal{F},P))$. This formalism extends naturally to
stochastic processes indexed by time or more general parameter sets. Let $%
L^{2}(\mathcal{F},\mathbb{R})$ be the Hilbert space with inner product $%
(x,y)_{2}=\mathbb{E}[x\cdot y]$, where $x,y\in L^{2}(\mathcal{F},\mathbb{R}%
). $ For any pair of random variable $(\xi ,\eta )\in $ $L^{2}(\mathcal{F},%
\mathbb{R})\times L^{2}(\mathcal{F},\mathbb{R}),$ we let $(\tilde{\xi},%
\tilde{\eta})$ be an independent copy of $(\xi ,\eta )$ defined on $(%
\widetilde{\Omega },\widetilde{\mathcal{F}},\widetilde{P})$. We consider the
product probability space $(\Omega \times \widetilde{\Omega },\mathcal{F}%
\otimes \widetilde{\mathcal{F}},\mathbb{P}\otimes \widetilde{P})$ and
setting $(\tilde{\xi},\tilde{\eta})\left( \omega ,\tilde{\omega}\right)
=(\xi \left( \tilde{\omega}\right) ,\eta \left( \tilde{\omega}\right) )$ for
any $\left( \omega ,\tilde{\omega}\right) \in \Omega \times \widetilde{%
\Omega }$.

Now, let us recall the differentiability with respect to probability
measures studied by Lions \cite{lps} to derive our main result. The main
idea is to identify a distribution $\mu \in \mathbb{P}_{2}(\mathbb{R}^{d})$
with a random variables $\eta \in \mathbb{L}^{2}(\mathcal{F};\mathbb{R}^{d})$
so that $\mu =\mathbb{P}_{\eta }.$ To be more precise, we assume that
probability space $(\Omega ,\mathcal{F},\mathcal{F}_{t},\mathbb{P})$ is rich
enough in the sense that for every $\mu \in \mathcal{P}_{2}(\mathbb{R}^{d}),$
there is a random variable $\eta \in \mathbb{L}^{2}(\mathcal{F};\mathbb{R}%
^{d})$\ such that $\mu =\mathbb{P}_{\eta }$ ($\mu =\mathcal{L}\left( \mu
\right) $)$.$ It is well known that the probability space $\left( \left[ 0,1%
\right] ,\mathcal{B}\left[ 0,1\right] ,\mathrm{d}x\right) ,$ where $\mathrm{d%
}x$ is the Borel measure has this property, see \cite{blj16}.

\begin{definition}[Lift function]
Let $f$ be a given function such that $f:\mathcal{P}_{2}(\mathbb{R}%
^{d})\rightarrow \mathbb{R}.$ We define the lift function $\tilde{f}:\mathbb{%
L}^{2}(\mathcal{F};\mathbb{R}^{d})\rightarrow \mathbb{R}$ such that $\tilde{f%
}(Z):=f(\mathbb{P}_{Z}),Z\in \mathbb{L}^{2}(\mathcal{F};\mathbb{R}^{d}).$
Clearly, the lift function $\tilde{f}$ of $f$, depends only on the law of $%
Z\in $ $\mathbb{L}^{2}(\mathcal{F};\mathbb{R}^{d})$ and is independent of
the choice of the representative $Z$.
\end{definition}

A function $f$ $:\mathcal{P}_{2}(\mathbb{R}^{d})\rightarrow \mathbb{R}$ is
said to be differentiable at $\mu _{0}\in \mathbb{P}_{2}(\mathbb{R}^{d})$ if
there exists $\eta _{0}\in $ $\mathbb{L}^{2}(\mathcal{F};\mathbb{R}^{d})$
with $\mu _{0}=\mathbb{P}_{\eta _{0}}$ such that its lift function $\tilde{f}
$ is Fr\'{e}chet differentiable at $\eta _{0}.$ More precisely, there exists
a continuous linear functional $D\tilde{f}(\eta _{0}):\mathbb{L}^{2}(%
\mathcal{F};\mathbb{R}^{d})\rightarrow \mathbb{R}$ such that%
\begin{eqnarray}
\tilde{f}(\eta _{0}+\xi )-\tilde{f}(\eta _{0}) &=&\left\langle D\tilde{f}%
(\eta _{0}),\xi \right\rangle +O(\left\Vert \xi \right\Vert _{2})
\label{2.1} \\
&=&D_{\xi }f(\mu _{0})+O(\left\Vert \xi \right\Vert _{2}),  \notag
\end{eqnarray}%
where $\left\langle \cdot ,\cdot \right\rangle $ is the dual product on $%
\mathbb{L}^{2}(\mathcal{F};\mathbb{R}^{d}),$ and we will refer to $D_{\xi
}f(\mu _{0})$ as the Fr\'{e}chet derivative of $f$ at $\mu _{0}$ in the
direction $\xi .$ In this case, we have%
\begin{equation*}
D_{\xi }f(\mu _{0})=\left\langle D\tilde{f}(\eta _{0}),\xi \right\rangle =%
\frac{d}{dt}\tilde{f}(\eta _{0}+t\xi )|_{t=0},\text{ with }\mu _{0}=\mathbb{P%
}_{\eta _{0}}.
\end{equation*}%
By applying the Riesz' representation theorem, there is a unique random
variable $z_{0}\in \mathbb{L}^{2}(\mathcal{F};\mathbb{R}^{d})$ such that $%
\left\langle D\tilde{f}(\eta _{0}),\xi \right\rangle =(z_{0},\xi )=\mathbb{E}%
\left[ (z_{0},\xi )\right] ,$ where $\xi \in \mathbb{L}^{2}(\mathcal{F};%
\mathbb{R}^{d}).$ It was shown, (see the works of {\cite{blj16}} and \cite%
{lps}) that there exists a Borel function $h\left[ \mu _{0}\right] \left(
\cdot \right) :\mathbb{R}^{d}\rightarrow \mathbb{R}^{d},$ depending only on
the law $\mu _{0}$ $=\mathbb{P}_{\eta _{0}}$ but not on the particular
choice of the representative $\eta _{0}$ such that $z_{0}=h\left[ \mu _{0}%
\right] (\eta _{0}).$

Thus, we can write (\ref{2.1}) as%
\begin{equation*}
f(\mathbb{P}_{\eta })-f(\mathbb{P}_{\eta _{0}})=(h\left[ \mu _{0}\right]
(\eta _{0}),\eta -\eta _{0})_{2}+O(\left\Vert \eta -\eta _{0}\right\Vert
_{2}),\forall \eta \in \mathbb{L}^{2}(\mathcal{F};\mathbb{R}^{d}).
\end{equation*}%
We denote $\partial _{\mu }f(\mathbb{P}_{\eta _{0}},x)=h\left[ \mu _{0}%
\right] (x),x\in \mathbb{R}^{d}.$ Moreover, we have the following
identities: $D\tilde{f}(\eta _{0})=z_{0}=h\left[ \mu _{0}\right] (\eta
_{0})=\partial _{\mu }f(\mathbb{P}_{\eta _{0}},\eta _{0}),$ and $D_{\xi
}f(\mu _{0})=\left\langle \partial _{\mu }f(\mathbb{P}_{\eta _{0}},\eta
_{0}),\xi \right\rangle ,$ where $\xi =\eta -\eta _{0}.$

\begin{remark}
We point out that for each $\mu \in \mathcal{P}_{2}(\mathbb{R}^{d}),$ $%
\partial _{\mu }f(\mathbb{P}_{\eta _{0}},\cdot )=h\left[ \mathbb{P}_{\eta }%
\right] (\cdot )$ is only defined in a $\mathbb{P}_{\eta }(\mathrm{d}x)$-a.s
sense, where $\mu =\mathbb{P}_{\eta }.$
\end{remark}

\begin{definition}[Space of differentiable functions in $\mathcal{P}_{2}(%
\mathbb{R}^{d})$]
\label{lions} We say that the function $f\in \mathbb{C}_{b}^{1,1}(\mathcal{P}%
_{2}(\mathbb{R}^{d}))$ if for all $\eta \in \mathbb{L}^{2}(\mathcal{F};%
\mathbb{R}^{d}),$ there exists a $\mathbb{P}_{\eta }$-modification of $%
\partial _{\mu }f(\mathbb{P}_{\eta },\cdot )$ such that $\partial _{\mu }f:%
\mathcal{P}_{2}(\mathbb{R}^{d})\times \mathbb{R}^{d}\rightarrow \mathbb{R}%
^{d}$ is bounded and Lipschitz continuous. That is for some $C>0$, it holds
that (i) $\left\vert \partial _{\mu }f(\mu ,x)\right\vert \leqslant C,$ $%
\forall \mu \in \mathcal{P}_{2}(\mathbb{R}^{d}),$ $\forall x\in \mathbb{R}%
^{d},$ and (ii) $\left\vert \partial _{\mu }f(\mu _{1},x_{1})-\partial _{\mu
}f(\mu _{2},x_{2})\right\vert \leqslant C(\mathcal{W}_{2}(\mu _{1},\mu
_{2})+\left\vert x_{1}-x_{2}\right\vert ),$ $\forall \mu _{1},\mu _{2}\in
\mathcal{P}_{2}(\mathbb{R}^{d}),$ $\forall x_{1},x_{2}\in \mathbb{R}^{d}.$
\end{definition}

Note that the version of $\partial _{\mu }f(\mathbb{P}_{\eta },\cdot ),$ $%
\eta \in \mathbb{L}^{2}(\mathcal{F};\mathbb{R}^{d})$ indicated in the above
definition is unique (see Remark 2.2 in \cite{blj16} for more information%
\footnote{%
By leveraging the property that Lipschitz functions that are defined almost
everywhere can be uniquely extended to the entire space, the measure
derivative is extended from the complement of the zero-measure set to the
entire measure space, state space, thereby resolving the issues of
\textquotedblleft incomplete definition and non-uniqueness\textquotedblright
.}).

Let $\Pi =\mathfrak{F}^{0,2}\left( \left[ 0,T\right] \times \mathbb{R\times }%
\mathcal{P}_{2}\left( \mathbb{R}\right) ;\mathbb{R}\right) $ be the vector
space of measurable functions $\varphi :\Omega \times \left[ 0,T\right]
\times \mathbb{R\times }\mathcal{P}_{2}\left( \mathbb{R}\right) \rightarrow
\mathbb{R}$ such that for all $\left( \omega ,t\right) \in \Omega \times %
\left[ 0,T\right] ,$ $x\rightarrow \varphi _{t}\left( x\right) $ is
continuously differentiable; There exists $L^{\varphi }>0$ such that for all
$\left( t,x,\mu \right) \in \left[ 0,T\right] \times \mathbb{R\times }%
\mathcal{P}_{2}\left( \mathbb{R}\right) $ and$,$ $\left\vert \varphi
_{t}\left( 0,\delta _{0}\right) \right\vert \leq L^{\varphi },$ $\left\vert
\left( \partial _{x}\varphi _{t}\right) \left( x,\mu \right) \right\vert
\leq L^{\varphi },$ $\left\vert \left( \partial _{\mu }\varphi _{t}\right)
\left( x,\mu \right) \left( x\right) \right\vert \leq L^{\varphi }$ (for
details of Lions derivative, please refer \cite{Lions2012})$.$ For any $%
\left( x,\mu \right) \in \mathbb{R\times }\mathcal{P}_{2}\left( \mathbb{R}%
\right) ,$ $\varphi \left( t,x,\mu \right) $ is $\mathcal{F}_{t}$-adapted.
The support of $\mu _{A}$ is $\mathbb{R}.$ Let $\mathcal{A}$ be a convex
subset of $\Pi $ representing the set of admissible closed-loop controls.

For each $\phi \in \mathcal{A},$ let $x^{\phi }$ be the associated wealth
process driven by the dynamics (\ref{sde1}):

\begin{equation}
\left\{
\begin{array}{lll}
\mathrm{d}x^{\phi }\left( t\right) & = & \left[ A\left( t\right) x^{\phi
}\left( t\right) +B\left( t\right) u^{\phi }\left( t\right) \right] \mathrm{d%
}t+\sum_{j=1}^{m}D_{j}\left( t\right) u^{\phi }\left( t\right) \mathrm{d}%
W^{j}\left( t\right) , \\
x^{\phi }\left( 0\right) & = & x_{0}>0,\text{ with }u^{\phi }\left( t\right)
=\phi \left( t,x^{\phi }\left( t\right) ,\mathbb{P}_{x^{\phi }\left(
t\right) }\right) ,\text{ }t\in \left[ 0,T\right] ,%
\end{array}%
\right.  \label{sde1}
\end{equation}%
where the initial wealth $x_{0}$ is a random variable with distribution
given by $\mu _{0}\in \mathcal{P}(\mathbb{R})$. Moreover, $A(t)=r\left(
t\right) $ is scalar, $B(t)=\left( b_{1}\left( t\right) -r\left( t\right)
,\ldots ,b_{m}\left( t\right) -r\left( t\right) \right) $ and $%
D_{j}(t)=\left( \sigma _{1j}\left( t\right) ,\ldots ,\sigma _{mj}\left(
t\right) \right) ,$ $(j=1,\cdots ,m)$ are vectors. In addition, we assume
that the matrix $\Pi \left( t\right) =\sum_{j=1}^{m}D_{j}(t)^{\top }D_{j}(t)$
is \textit{non-singular}.

\noindent \textbf{Problem (A). }The investor is to seek an optimal $\phi
^{\ast }\in \mathcal{A}$ that minimizes the quadratic (terminal) cost
function%
\begin{equation}
J(\phi )=\mathbb{E}\left[ \frac{\lambda }{2}\mathbb{D}\left( x^{\phi }\left(
T\right) \right) -\mathbb{E}\left[ x^{\phi }\left( T\right) \right] +%
\mathcal{W}_{2}\left( \mathbb{P}_{x^{\phi }\left( T\right) },\mu _{A}\right)
^{2}\right] ,  \label{cost1}
\end{equation}%
where $\lambda >0$. In the absence of stochastic factors in the price
dynamics---thus within the framework of a complete market model---and when $%
\lambda $ is constant, the above criterion reduces to the classical
mean-variance portfolio selection problem, as analyzed, for instance, by Li
and Zhou (2000) \cite{lz2000}.

As for $\mu _{A}$, consider the same stock market framework described in the
previous section. The market index process is defined as a weighted linear
combination of component stocks:
\begin{equation}
I(t)=\sum_{j=1}^{m}\alpha _{j}S_{j}(t),\quad I(0)=I_{0},
\label{eq:index_def}
\end{equation}
where $\alpha _{j}$ denotes the fixed weight share of the $j$-th component
stock in the market index. Then we set
\begin{equation*}
\mu _{A}=\mathbb{P}_{I(T)}.
\end{equation*}
We adopt the Wasserstein-$2$ metric $\mathcal{W}_{2}$ to quantify the
discrepancy between the terminal wealth distribution $\mathcal{L}(X_{T})$ of
the self-financing portfolio and the target market index distribution $\mathcal{L}(I_{T})$. We elaborate its mathematical rationality and financial
advantages as follows.

\subparagraph{Mathematical rationality of $\mathcal{W}_{2}$ distance.}

First, $\mathcal{W}_{2}$ defines a complete separable metric on the space $\mathcal{P}_{2}(\mathbb{R})$ of probability measures with finite second
moments, satisfying all four metric axioms. The identity condition $W_{2}(\mu ,\nu )=0\iff \mu =\nu $ guarantees that the loss vanishes if and
only if the portfolio perfectly replicates the index's terminal
distribution, which provides a well-posed objective for optimization.

Second, the Wasserstein-$2$ distance serves as the fundamental lower bound
of mean squared tracking error:
\begin{equation*}
\mathcal{W}_{2}^{2}\big(\mathcal{L}(X_{T}),\mathcal{L}(I_{T})\big)\leq
\mathbb{E}|X_{T}-I_{T}|^{2}.
\end{equation*}
Unlike the path-wise mean squared error which relies on a specific joint
coupling of $X_{T}$ and $I_{T}$, $W_{2}$ eliminates the interference of
trajectory correlation and only reflects the intrinsic gap between two
marginal distributions.

Third, convergence under $\mathcal{W}_{2}$ metric simultaneously implies
weak convergence of measures and convergence of second moments. Minimizing $%
\mathcal{W}_{2}$ thus jointly aligns the expectation, volatility, skewness
and tail risk of terminal payoffs, while the classical quadratic loss only
controls the first two moments and fails to capture distributional tail
discrepancies.

Moreover, the Wasserstein-$2$ metric equips $\mathcal{P}_{2}(\mathbb{R})$
with a Riemannian manifold structure, which is compatible with the $\mathbb{C%
}_{b}^{1,1}$ smoothness framework and Lions' measure derivative $\partial
_{\mu }f$ introduced in Definition \ref{lions} Lipschitz regularization,
extension arguments and uniqueness analysis of measure derivatives can be
directly applied under the $\mathcal{W}_{2}$ topology, forming a
self-consistent theoretical system.

\subparagraph{Financial advantages for index tracking problems}

1. \textit{Consistent with the core demand of passive index replication}:
The essential goal of index tracking is to fully replicate the return
distribution of the benchmark index, including both normal fluctuation and
extreme tail events. The $\mathcal{W}_{2}$ metric penalizes all forms of
distributional mismatch, avoiding the defect of mean squared error that
portfolios with identical expectation and variance may carry completely
different crash risks.

2. \textit{Clear economic interpretation}: From the optimal transport
perspective, $\mathcal{W}_{2}^{2}(\mu ,\nu )$ represents the minimal
quadratic cost to transform the portfolio wealth distribution into the index
distribution, which acts as a standardized, interpretable performance metric
for tracking error evaluation.

3. \textit{Stability under model perturbations}: The Wasserstein metric
depends continuously on model parameters such as index weights and stock
diffusion coefficients. Small perturbations of market parameters induce
continuous changes of the objective functional, ensuring numerical stability
during optimal portfolio optimization.

4. \textit{Wide compatibility with stochastic control frameworks}: The $%
\mathcal{W}_{2}$ loss can be extended to infinite-horizon discounted
continuous-time tracking problems, compatible with SDE dynamics, mean-field
approximation and slow-fast scale separation analysis for diffusion
processes.

\subparagraph{Comparison with classical mean squared tracking loss}

The standard quadratic expectation loss only captures the first two moments
of joint payoffs and lacks support for measure-space differential calculus.
In contrast, the Wasserstein-$2$ distance achieves full characterization of
distributional discrepancy, supports rigorous smoothness analysis on measure
spaces, and better matches the practical risk control requirements of
passive index funds.
%where $\lambda $ is a positive $\mathcal{F}_{T}$-measurable random variable.
%In the absence of random factors in the dynamics of the price process, hence
%in a complete market model, and when $\lambda $ is constant, the above
%criterion reduces to the classical mean-variance portfolio selection, as
%studied e.g. in (Li and Zhou 2000).

Given $\mu _{A}\in \mathcal{P}_{2}\left( \mathbb{R}\right) $ for instance
above$,$ define $g_{A}:\mathcal{P}_{2}\left( \mathbb{R}\right) \rightarrow
\mathbb{R}$ by%
\begin{equation*}
g_{A}\left( \mu \right) =\frac{\lambda }{2}\mathbb{V}\text{ar}\left( \mu
\right) -\bar{\mu}+\mathcal{W}_{2}\left( \mu ,\mu _{A}\right) ^{2},
\end{equation*}%
where
\begin{equation*}
\bar{\mu}=\int_{\mathbb{R}}x\mu \left( \mathrm{d}x\right) ,\text{ }\bar{\mu}%
_{2}=\int_{\mathbb{R}}x^{2}\mu \left( \mathrm{d}x\right) ,\text{ }\mathbb{V}%
\text{ar}\left( \mu \right) =\bar{\mu}_{2}-\bar{\mu}^{2},\text{ }\mu \in
\mathcal{P}_{2}\left( \mathbb{R}\right) .
\end{equation*}%
Immediately, the cost functional (\ref{cost1}) can be expressed as $J(\phi )=%
\mathbb{E}\left[ g_{A}\left( \mathbb{P}_{x\left( T\right) }\right) \right] $.

We now briefly introduce the well-known Kantorovich duality theorem in
optimal transport theory.

\begin{theorem}
\label{the1}Let $(E,\rho )$ be a Polish space, $\mathscr{E}$ be the $\sigma $%
-algebra on $E$, and let $\mathcal{P}_{p}(E)$ denote the space of Borel
probability measures on $E$ with finite $p$-th moment where $p\geq 1$. The $%
p $-Wasserstein distance is defined by
\begin{equation}
\mathcal{W}_{p}^{\rho }(\mu ,\nu )\triangleq \inf_{\pi \in \mathscr{L}(%
%TCIMACRO{\U{3bc} }%
%BeginExpansion
\mu
%EndExpansion
,\nu )}\left\{ \int_{E\times E}\rho ^{p}(x,y)\mathrm{d}\pi (x,y)\right\}
^{1/p},\text{ }\forall \mu ,\nu \in \mathcal{P}_{p}(E)  \label{w1},
\end{equation}%
where
\begin{equation*}
\mathscr{C}(\mu ,\nu )=\left\{ \pi \in \mathcal{P}(E\times E)\mid \forall
A\in \mathscr{E},\ \pi (A\times E)=\mu (A),\ \pi (E\times A)=\nu (A)\right\}
.
\end{equation*}%
Then there exists a dual representation of the $p$-Wasserstein distance via
the following identity:
\begin{equation}
\mathcal{W}_{p}^{\rho }(\mu ,\nu )^{p}=\sup_{(f,g)\in \mathscr{F}_{\mathrm{%
Lip}}}\left\{ \int_{E}fd\mu +\int_{E}gd\nu \right\} ,  \label{w2}
\end{equation}%
where
\begin{equation*}
\mathscr{F}_{\mathrm{Lip}}\triangleq \left\{ \left( f,g\right) \mid f,g\text{
are Lipschitz continuous satisfying }f\left( x\right) +g\left( y\right) \leq
\rho ^{p}\left( x,y\right) \right\} .
\end{equation*}
\end{theorem}

%Theorem 1 (Rachev \cite{r85}, page 649) states above results when the
%supremum in (\ref{w2}) are taken over $\mathscr{F}_{\mathrm{Lip}}^{b}$
%denoting the space of Lipschitz continuous and bounded functions. In
%Proposition 5.3 (Carmona and Delarue \cite{cdi18}, page 354), the supremum
%is taken over all the real valued bounded continuous functions. While in
%Theorem 5.10 (Villani \cite{v09}, page 57), the supremum is taken over all
%the real valued $L^{1}$-integrable functions. Due to the real application,
%we remove the condition \emph{boundedness} in Theorem \ref{the1}. The proof can be found in
%Appendix \ref{app}.
In the dual formulation of the Monge-Kantorovich problem (MKP), each symbol
carries a well-defined physical or economic interpretation, which is
essential for understanding the underlying constraints. Specifically:

\begin{itemize}
\item $x$: an arbitrary point in the source space, representing a location
of initial mass distribution (e.g., the origin of goods or material);

\item $y$: an arbitrary point in the target space, representing a
destination in the final mass distribution (e.g., the delivery site of
goods);

\item $\rho (x,y)$: the cost function in the primal transport problem,
denoting the actual cost of transporting one unit of mass from $x$ to $y$,
which may correspond to distance, monetary expense, or time;

\item $f(x)$: the value function associated with the source point $x$,
interpreted as the \textquotedblleft implicit value\textquotedblright\ of
acquiring one unit of mass at $x$, or equivalently, the minimum
\textquotedblleft compensation value\textquotedblright\ required to extract one unit of mass from $x$;

\item $g(y)$: the value function associated with the target point $y$,
interpreted as the \textquotedblleft implicit value\textquotedblright\ of
receiving one unit of mass at $y$, or equivalently, the maximum \textquotedblleft payment
value\textquotedblright\ one would be willing to offer to deliver mass to $y$.

\item The constraint $f(x)+g(y)\leq \rho (x,y)$ can be literally interpreted
as follows: the sum of the value of acquiring mass at $x$ and the value of
delivering mass at $y$ must not exceed the actual transportation cost from $%
x $ to $y$. The core intuition behind this constraint is to prevent
``arbitrage opportunities" and ensure the economic and physical consistency
of the dual solution. This condition eliminates non-physical or economically
infeasible value assignments by enforcing feasibility through a no-arbitrage
principle. To illustrate, suppose there exists a pair $(x_{0},y_{0})$ such
that $f(x_{0})+g(y_{0})>\rho (x_{0},y_{0})$. This would imply that the
combined value of extracting mass at $x_{0}$ and delivering it to $y_{0}$
exceeds the cost of transportation---enabling a risk-free profit. In such a
scenario, one could generate unbounded gains by repeatedly transporting mass
from $x_{0}$ to $y_{0}$, which contradicts the fundamental assumption that
transportation incurs a real, irreducible cost. Hence, the constraint
ensures that the dual variables reflect realistic economic valuations and
that the optimal solution remains consistent with the physical constraints
of the transport process.
\end{itemize}

Theorem 1 (Rachev \cite{r85}, p. 649) presents the aforementioned results
when the supremum in (\ref{w2}) is taken over $\mathscr{F}_{\mathrm{Lip}%
}^{b} $, the space of Lipschitz continuous and bounded real-valued
functions. In Proposition 5.3 (Carmona and Delarue \cite{cdi18}, p. 354),
the supremum is taken over the set of all bounded continuous real-valued
functions. In contrast, Theorem 5.10 (Villani \cite{v09}, p. 57) considers
the supremum over all real-valued $L^{1}$-integrable functions. 

%For
%practical applications, we relax the boundedness condition imposed in
%Theorem \ref{the1}. The corresponding proof is provided in Appendix \ref{app}%
%.

Rademacher's theorem (see \cite{eg15}, Theorem 2, page 81) implies that $f$
and $g$ possess the derivatives almost everywhere on $\mathbb{R}$.
Accordingly, $\partial _{x}f^{\ast }$ exists almost everywhere. It should be
noted that, under the assumption, the support of $\mu _{A}$ is assumed to be
$\mathbb{R}$, the uniqueness of the Kantorovich potentials up to additive
constants can be established by applying Proposition 7.18 in \cite{san15}
with minor modifications. This uniqueness property is subsequently utilized
in the following proofs.%
%Note that, as we assume the support of $\mu _{A}$ to be $\mathbb{R}$ in
%Assumption ?, we can establish the uniqueness of the Kantorovich potentials
%up to additive constants by applying (\cite{san15}, Proposition 7.18) with
%minor modifications and that the uniqueness is used in the following proofs.

This result reformulates the primal optimal transport problem---originally
defined as an infimum over couplings---into a dual maximization problem over
suitable function pairs. Such a duality is fundamental in optimal transport
theory, providing a powerful analytical tool for studying the properties of
Wasserstein distances and enabling practical numerical approximations.

%The plan of this paper is orgonized as follows: After notations in Section \ref{sect2}, we derive the optimal feedback strategy for
%problem (A) by Pontryagin's maximum principle
%and Kantorovich duality theorem. Due to the variance appearing in cost
%functional (\ref{cost1}), the feedback control depends on the state and
%probability measure. When Kantorovich potential turns to be linear function,
%is is possible to seek the the corresponding Riccati equations.

\section{Main results}

\label{sect3}

%In this section, we derive the optimal feedback strategy for
%problem (A). We adopt the Pontryagin's maximum principle by means of BSDE
%and Kantorovich duality theorem. Due to the variance appearing in cost
%functional (\ref{cost1}), the feedback control depends on the state and
%probability measure. When Kantorovich potential turns to be linear function,
%is is possible to seek the the corresponding Riccati equations.

In this section, we derive the optimal feedback control strategy for Problem
(A). Our approach combines the stochastic maximum principle, formulated via
BSDEs with the Kantorovich duality theorem from optimal transport theory.
Owing to the presence of a variance term in the cost functional (\ref{cost1}%
), the resulting optimal feedback law depends jointly on the current state
and the law (i.e., probability distribution) of the state process. In the
special case where the Kantorovich potential is affine-i.e., linear up to an
additive constant, the associated optimality conditions reduce to a system
of coupled Riccati-type differential equations, which admit a well-posed
solution under standard regularity and convexity assumptions.

We begin by considering the first-order variational equation associated with
the underlying system.

\begin{lemma}
\label{l1} For $\phi ^{\ast }\in \mathcal{A}$, let $x^{\phi ^{\ast }}$ be
the wealth process satisfying (\ref{sde1}). For each $\phi ^{\prime }\in
\mathcal{A}$, let $y^{\phi ^{\ast },\phi ^{\prime }}$ be the solution to the
following dynamics: for all $t\in \lbrack 0;T]$,%
\begin{equation}
\left\{
\begin{array}{rcl}
\mathrm{d}y^{\phi ^{\ast },\phi ^{\prime }}\left( t\right) & = & \Big \{%
\left( A\left( t\right) +B\left( t\right) \partial _{x}\phi ^{\ast }\left(
t,x^{\phi ^{\ast }}\left( t\right) ,\mathbb{P}_{x^{\phi ^{\ast }}\left(
t\right) }\right) \right) y^{\phi ^{\ast },\phi ^{\prime }}\left( t\right)
\\
&  & +\mathbb{\hat{E}}\left[ B\left( t\right) \partial _{\mu }\phi ^{\ast
}\left( t,x^{\phi ^{\ast }}\left( t\right) ,\mathbb{P}_{x^{\phi ^{\ast
}}\left( t\right) };\hat{x}^{\phi ^{\ast }}\left( t\right) \right) \hat{y}%
^{\phi ^{\ast },\phi ^{\prime }}\left( t\right) \right] \\
&  & +B\left( t\right) \phi ^{\prime }\left( t,x^{\phi ^{\ast }}\left(
t\right) ,\mathbb{P}_{x^{\phi ^{\ast }}\left( t\right) }\right) \Big \}%
\mathrm{d}t \\
&  & +\Big \{\left( \sum_{j=1}^{m}D_{j}\left( t\right) \partial _{x}\phi
^{\ast }\left( t,x^{\phi ^{\ast }}\left( t\right) ,\mathbb{P}_{x^{\phi
^{\ast }}\left( t\right) }\right) \right) y^{\phi ^{\ast },\phi ^{\prime
}}\left( t\right) \\
&  & +\mathbb{\hat{E}}\left[ \sum_{j=1}^{m}D_{j}\left( t\right) \partial
_{\mu }\phi ^{\ast }\left( t,x^{\phi ^{\ast }}\left( t\right) ,\mathbb{P}%
_{x^{\phi ^{\ast }}\left( t\right) };\hat{x}^{\phi ^{\ast }}\left( t\right)
\right) \hat{y}^{\phi ^{\ast },\phi ^{\prime }}\left( t\right) \right] \\
&  & +\sum_{j=1}^{m}D_{j}\left( t\right) \phi ^{\prime }\left( t,x^{\phi
^{\ast }}\left( t\right) ,\mathbb{P}_{x^{\phi ^{\ast }}\left( t\right)
}\right) \Big \}\mathrm{d}W^{j}\left( t\right) , \\
\hat{y}^{\phi ^{\ast },\phi ^{\prime }}\left( 0\right) & = & 0,\text{ }t\in %
\left[ 0,T\right] .%
\end{array}%
\right.  \label{v1}
\end{equation}%
Then the following estimation holds:%
\begin{equation*}
\lim_{\varepsilon \rightarrow 0}\mathbb{E}\left[ \sup_{t\in \left[ 0,T\right]
}\left\vert \frac{1}{\varepsilon }\left( x^{\phi ^{\varepsilon }}\left(
t\right) -x^{\phi ^{\ast }}\left( t\right) \right) -y^{\phi ^{\ast },\phi
^{\prime }}\left( t\right) \right\vert ^{2}\right] =0,
\end{equation*}%
where $\phi ^{\varepsilon }=\phi ^{\ast }+\varepsilon \phi ^{\prime }$ for
all $\varepsilon \in (0,1)$.
\end{lemma}

Lemma \ref{l1} implies that $y^{\phi ^{\ast },\phi ^{\prime }}$ is the
derivative of the wealth process $x^{\phi ^{\ast }}$ when the investor
varies her policy in the direction $\phi ^{\prime }$.

The proof is obtained by integrating Pontryagin's maximum principle for
stochastic control problems and the Kantorovich duality theorem. Although
the idea is based on the proof of stochastic Pontryagin's principle, the
proof requires handling the Wasserstein distance to obtain the terminal
condition (10). This study handles this by utilizing the Kantorovich duality
theorem.

We first introduce notations before providing the proof. The optimal
closed-loop strategy $\phi ^{\ast }=\left\{ \phi ^{\ast }\left( t,x^{\phi
^{\ast }}\left( t\right) ,\mathbb{P}_{x^{\phi ^{\ast }}\left( t\right)
}\right) \right\} _{0\leq t\leq T}$ determines the distribution $\mu
_{T}^{\ast }$ of the terminal state $x^{\phi ^{\ast }}\left( T\right) $ and
the squared Wasserstein distance $\mathcal{W}_{2}^{\rho }(\mu ,\nu )^{2}$ in
the cost functional (6). According to Theorem 1, there exist Kantorovich
potentials $\varphi ^{\ast }$ and $\psi ^{\ast }$ such that the following
holds:

\begin{equation}
\mathcal{W}_{2}^{\rho }(\mu _{T}^{\ast },\mu _{A})^{2}=\int_{\mathbb{R}%
}\varphi ^{\ast }\mathrm{d}\mu _{T}^{\ast }-\int_{\mathbb{R}}\psi ^{\ast }%
\mathrm{d}\mu _{A}.  \label{cost2}
\end{equation}%
Immediately, the cost functional (\ref{cost1}) can be rewritten as
\begin{equation}
J(\phi ^{\ast } )=\mathbb{E}\left[ g_{A}\left( \mathbb{P}_{x^{\phi ^{\ast }}\left(
T\right) }\right) \right] =\mathbb{E}\left[ \frac{\lambda }{2}\mathbb{V}%
\text{ar}\left( \mu _{T}^{\ast }\right) -\bar{\mu}_{T}^{\ast }+\int_{\mathbb{%
R}}\varphi ^{\ast }\mathrm{d}\mu _{T}^{\ast }-\int_{\mathbb{R}}\psi ^{\ast }%
\mathrm{d}\mu _{A}\right] .  \label{cost3}
\end{equation}%
First, we show that the G\^{a}teaux derivative of the cost function (\ref%
{cost1}) becomes (letting $\phi ^{\prime \prime }=\phi ^{\prime }-\phi
^{\ast }$)%
\begin{eqnarray}
0 &\leq &\lim_{\varepsilon \rightarrow 0}\frac{J(\phi ^{\ast }+\varepsilon
\left( \phi ^{\prime \prime }\right) )-J\left( \phi ^{\ast }\right) }{%
\varepsilon }  \notag \\
&=&\lim_{\varepsilon \rightarrow 0}\frac{\mathbb{E}\left[ g_{A}\left(
\mathbb{P}_{x^{\phi ^{\prime \prime }}\left( T\right) }\right) -g_{A}\left(
\mathbb{P}_{x^{\phi ^{\ast }}\left( T\right) }\right) \right] }{\varepsilon }
\notag \\
&=&\mathbb{E}\left\{ \mathbb{\hat{E}}\left[ \partial _{\mu }g_{A}\left(
\mathbb{P}_{x^{\phi ^{\ast }}\left( T\right) };\hat{x}^{\phi ^{\ast }}\left(
T\right) \right) \hat{y}^{\phi ^{\ast },\phi ^{\prime \prime }}\left(
T\right) \right] \right\}  \notag \\
&=&\mathbb{E}\left\{ \mathbb{\hat{E}}\left[ \left( \lambda \left( \hat{x}%
^{\phi ^{\ast }}\left( T\right) -\mathbb{E}\left[ \hat{x}^{\phi ^{\ast
}}\left( T\right) \right] \right) -1+\partial _{x}\varphi ^{\ast }\left(
\hat{x}^{\phi ^{\ast }}\left( T\right) \right) \right) \hat{y}^{\phi ^{\ast
},\phi ^{\prime \prime }}\left( T\right) \right] \right\} .  \label{var1}
\end{eqnarray}%
We introduce the adjoint equations involved in the stochastic maximum
principle for our control problem%
\begin{equation}
\left\{
\begin{array}{rcl}
-\mathrm{d}p\left( t\right) & = & \Big \{\left( A\left( t\right) +B\left(
t\right) \partial _{x}\phi ^{\ast }\left( t,x^{\phi ^{\ast }}\left( t\right)
,\mathbb{P}_{x^{\phi ^{\ast }}\left( t\right) }\right) \right) p\left(
t\right) \\
&  & +\mathbb{\hat{E}}\left[ B\left( t\right) \partial _{\mu }\phi ^{\ast
}\left( t,\hat{x}^{\phi ^{\ast }}\left( t\right) ,\mathbb{P}_{x^{\phi ^{\ast
}}\left( t\right) };x^{\phi ^{\ast }}\left( t\right) \right) \hat{p}\left(
t\right) \right] \\
&  & +\sum_{j=1}^{m}D_{j}\left( t\right) \partial _{x}\phi ^{\ast }\left(
t,x^{\phi ^{\ast }}\left( t\right) ,\mathbb{P}_{x^{\phi ^{\ast }}\left(
t\right) }\right) q_{j}\left( t\right) \\
&  & +\mathbb{\hat{E}}\left[ \sum_{j=1}^{m}D_{j}\left( t\right) \partial
_{\mu }\phi ^{\ast }\left( t,\hat{x}^{\phi ^{\ast }}\left( t\right) ,\mathbb{%
P}_{x^{\phi ^{\ast }}\left( t\right) };x^{\phi ^{\ast }}\left( t\right)
\right) \hat{q}_{j}\left( t\right) \right] \Big \}\mathrm{d}t \\
&  & -\sum_{j=1}^{m}q_{j}\left( t\right) \mathrm{d}W^{j}\left( t\right) , \\
p\left( T\right) & = & \partial _{\mu }g_{A}\left( \mathbb{P}_{x^{\phi
^{\ast }}\left( T\right) };x^{\phi ^{\ast }}\left( T\right) \right) .%
\end{array}%
\right.  \label{adj1}
\end{equation}%
Clearly, under assumption (H1), the adjoint equation (\ref{adj1}) admits one
and only one $\mathcal{F}_{t}$-adapted strong solution $\left( p\left( \cdot
\right) ,q\left( \cdot \right) \right) \in L_{\mathcal{F}}^{2}\left( \left[
0,T\right] ;\mathbb{R}\right) \times L_{\mathcal{F}}^{2}\left( \left[ 0,T%
\right] ;\mathbb{R}^{m}\right) $ (see the work of Buckdahn et al \cite{blj16}%
) such that $\mathbb{E}\left[ \sup_{0\leq t\leq T}\left\vert p\left(
t\right) \right\vert ^{2}+\int_{0}^{T}\left\vert q\left( t\right)
\right\vert ^{2}\mathrm{d}t\right] <+\infty .$

\begin{theorem}
\label{the2} Suppose Assumptions 4.1 and 4.2 hold. Let $\phi ^{\ast }\in
\mathcal{A}$ be an optimal closed-loop strategy and $x^{\phi ^{\ast }}$ be
the corresponding trajectory, then we have for any $\phi \in \mathcal{A}$,%
\begin{equation}
\mathbb{E}\left[ \int_{0}^{T}\left( p\left( t\right) B\left( t\right)
+\sum_{j=1}^{m}q_{j}\left( t\right) D_{j}\left( t\right) \right) \phi
^{\prime }\left( t,x^{\phi ^{\ast }}\left( t\right) ,\mathbb{P}_{x^{\phi
^{\ast }}\left( t\right) }\right) \mathrm{d}t\right] \geq 0,  \label{r1}
\end{equation}%
where $(p,q)$ is the solution of the adjoint Eq. (\ref{adj1}).
\end{theorem}

\begin{proof}

Applying It\^{o}'s formula to $p\left( t\right) \cdot y^{\phi ^{\ast },\phi
^{\prime }}\left( t\right) $ on $\left[ 0,T\right] $ and from Fubini's
theorem, we obtain%
\begin{eqnarray*}
&&\mathbb{E}\left[ p\left( T\right) \cdot y^{\phi ^{\ast },\phi ^{\prime
}}\left( T\right) \right] \\
&=&\mathbb{E}\left[ \int_{0}^{T}B\left( t\right) \phi ^{\prime }\left(
t,x^{\phi ^{\ast }}\left( t\right) ,\mathbb{P}_{x^{\phi ^{\ast }}\left(
t\right) }\right) p\left( t\right) +\sum_{j=1}^{m}D_{j}\left( t\right) \phi
^{\prime }\left( t,x^{\phi ^{\ast }}\left( t\right) ,\mathbb{P}_{x^{\phi
^{\ast }}\left( t\right) }\right) q_{j}\left( t\right) \mathrm{d}t\right] .
\end{eqnarray*}%
Note that
\begin{eqnarray*}
\mathbb{E}\left[ p\left( T\right) \cdot y^{\phi ^{\ast },\phi ^{\prime
}}\left( T\right) \right] &=&\mathbb{E}\left\{ \mathbb{\hat{E}}\left[
\partial _{\mu }g_{A}\left( \mathbb{P}_{x^{\phi ^{\ast }}\left( T\right)
};x^{\phi ^{\ast }}\left( T\right) \right) \right] \cdot y^{\phi ^{\ast
},\phi ^{\prime }}\left( T\right) \right\} \\
&=&\mathbb{E}\left\{ \mathbb{\hat{E}}\left[ \partial _{\mu }g_{A}\left(
\mathbb{P}_{x^{\phi ^{\ast }}\left( T\right) };\hat{x}^{\phi ^{\ast }}\left(
T\right) \right) \hat{y}^{\phi ^{\ast },\phi ^{\prime \prime }}\left(
T\right) \right] \right\} .
\end{eqnarray*}%
Therefore, from (\ref{var1}), we have%
\begin{equation*}
0\leq \mathbb{E}\left[ \int_{0}^{T}p\left( t\right) B\left( t\right) \phi
^{\prime }\left( t,x^{\phi ^{\ast }}\left( t\right) ,\mathbb{P}_{x^{\phi
^{\ast }}\left( t\right) }\right) +\sum_{j=1}^{m}q_{j}\left( t\right)
D_{j}\left( t\right) \phi ^{\prime }\left( t,x^{\phi ^{\ast }}\left(
t\right) ,\mathbb{P}_{x^{\phi ^{\ast }}\left( t\right) }\right) \mathrm{d}t%
\right] .
\end{equation*}%
The proof is thus complete.
\end{proof}

Next, we proceed to establish the stochastic Riccati equations and derive
the optimal closed-loop control strategy. For simplicity, assume that $%
\varphi ^{\ast }(x)=\alpha x,$ for some $\alpha \in \mathbb{R}$. Then, it
follows that $p(T)=\lambda x^{\phi ^{\ast }}\left( T\right) -\lambda \mathbb{%
E}\left[ x^{\phi ^{\ast }}\left( T\right) \right] -1+\alpha $.

\begin{remark}
This remark investigates the admissibility of explicit candidate potentials $%
(f,g)\in \mathscr{F}_{\mathrm{Lip}}$ for the $2$-Wasserstein duality. We
restrict to the one-dimensional Euclidean case $E=\mathbb{R}$ with standard
metric $\rho (x,y)=|x-y|$, subject to the fundamental constraint: A pair $%
(f,g)$ belongs to $\mathscr{F}_{\mathrm{Lip}}$ if and only if it satisfies
two indispensable requirements: (i) global Lipschitz continuity of both $f$
and $g$; (ii) the above pointwise inequality holds universally over $x,y\in
\mathbb{R}$. We now verify that linear functions are valid admissible
candidates for $p=2$, which resolves the confusion whether Lipschitz linear
potentials belong to $\mathscr{F}_{\mathrm{Lip}}$. We take the linear ansatz
$f(x)=kx$ for some constant slope $k\in \mathbb{R}$. For arbitrary fixed $%
y\in \mathbb{R}$, the duality constraint yields
\begin{equation*}
g(y)\leq \inf_{x\in \mathbb{R}}\big\{\rho ^{2}(x,y)-f(x)\big\}=\inf_{x\in
\mathbb{R}}\big\{(x-y)^{2}-kx\big\}.
\end{equation*}%
Direct calculus computation for the quadratic function $\phi
_{y}(x)=(x-y)^{2}-kx$ gives the global minimizer $x^{\ast }=y+\frac{k}{2}$.
Substitute back to derive the tightest feasible $g$:
\begin{equation*}
g(y)=-ky-\frac{k^{2}}{4}.
\end{equation*}%
Feasibility Verification (1) Global Lipschitz condition: Both potentials are
affine functions. Direct computation shows
\begin{equation*}
|f(x_{1})-f(x_{2})|=|k|\cdot |x_{1}-x_{2}|,\quad
|g(y_{1})-g(y_{2})|=|k|\cdot |y_{1}-y_{2}|,
\end{equation*}%
so $f,g$ are globally Lipschitz continuous with identical Lipschitz constant
$|k|$. (2) Pointwise duality constraint: Substitute the potential pair $%
(f,g)=(kx,-ky-\frac{k^{2}}{4})$. Let $t=x-y$, we obtain
\begin{equation*}
f(x)+g(y)=k(x-y)-\frac{k^{2}}{4}=kt-\frac{k^{2}}{4}\leq t^{2}=\rho ^{2}(x,y).
\end{equation*}%
The inequality holds trivially since $t^{2}-kt+\frac{k^{2}}{4}=(t-\frac{k}{2}%
)^{2}\geq 0$ for all $t\in \mathbb{R}$. Special Case: Take unit slope $k=1$,
we recover the explicit feasible pair
\begin{equation*}
f(x)=x,\qquad g(y)=-y-\frac{1}{4},
\end{equation*}%
which is the simplest non-trivial linear admissible pair for $2$-Wasserstein
duality. Notably, such linear pairs are legitimate elements in $\mathscr{F}_{%
\mathrm{Lip}}$, but they cannot achieve the supremum value unless the
optimal transport cost degenerates to a quadratic separable form.
\end{remark}

We introduce the following three classes of ODEs:
\begin{equation*}
\left\{
\begin{array}{rcl}
\mathrm{d}\Lambda \left( t\right) & = & \dot{\Lambda}\left( t\right) \mathrm{%
d}t,\text{ }0\leq t\leq T,\text{ }\Lambda \left( T\right) =\lambda , \\
\mathrm{d}\chi \left( t\right) & = & \dot{\chi}\left( t\right) \mathrm{d}t,%
\text{ }0\leq t\leq T,\text{ }\chi \left( T\right) =-\lambda , \\
\mathrm{d}\varphi \left( t\right) & = & \dot{\varphi}\left( t\right) \mathrm{%
d}t,\text{ }0\leq t\leq T,\text{ }\varphi \left( T\right) =\alpha -1,%
\end{array}%
\right.
\end{equation*}%
for some processes $\Lambda ,\chi ,\varphi ,\dot{\Lambda},\dot{\chi},$ and $%
\dot{\varphi}$ with values in $\mathbb{R}$, $Z^{\Lambda },$ $Z^{\Lambda },$
and $Z^{\varphi }$ with values in $\mathbb{R}^{m}$.

We conjecture that
\begin{equation*}
p\left( t\right) =\Lambda \left( t\right) x^{\phi ^{\ast }}(t)+\chi \left(
t\right) \mathbb{E}\left[ x^{\phi ^{\ast }}\left( t\right) \right] +\varphi
\left( t\right) .
\end{equation*}%
Then,
\begin{eqnarray}
&&-\Bigg \{\left( A\left( t\right) +B\left( t\right) \partial _{x}\phi
^{\ast }\left( t,x^{\phi ^{\ast }}\left( t\right) ,\mathbb{P}_{x^{\phi
^{\ast }}\left( t\right) }\right) \right) p\left( t\right)  \notag \\
&&+\mathbb{\hat{E}}\left[ B\left( t\right) \partial _{\mu }\phi ^{\ast
}\left( t,\hat{x}^{\phi ^{\ast }}\left( t\right) ,\mathbb{P}_{x^{\phi ^{\ast
}}\left( t\right) };x^{\phi ^{\ast }}\left( t\right) \right) \hat{p}\left(
t\right) \right]  \notag \\
&&+\sum_{j=1}^{m}D_{j}\left( t\right) \partial _{x}\phi ^{\ast }\left(
t,x^{\phi ^{\ast }}\left( t\right) ,\mathbb{P}_{x^{\phi ^{\ast }}\left(
t\right) }\right) q_{j}\left( t\right)  \notag \\
&&+\mathbb{\hat{E}}\Bigg [\sum_{j=1}^{m}D_{j}\left( t\right) \partial _{\mu
}\phi ^{\ast }\left( t,\hat{x}^{\phi ^{\ast }}\left( t\right) ,\mathbb{P}%
_{x^{\phi ^{\ast }}\left( t\right) };x^{\phi ^{\ast }}\left( t\right)
\right) \hat{q}_{j}\left( t\right) \Bigg ]\Bigg \}\mathrm{d}%
t+\sum_{j=1}^{m}q_{j}\left( t\right) \mathrm{d}W^{j}\left( t\right)  \notag
\\
&=&\mathrm{d}p\left( t\right)  \notag \\
&=&\Lambda \left( t\right) \left( \left[ A\left( t\right) x^{\phi ^{\ast
}}\left( t\right) +B\left( t\right) u^{\phi ^{\ast }}\left( t\right) \right]
\mathrm{d}t+\sum_{j=1}^{m}D_{j}\left( t\right) u^{\phi ^{\ast }}\left(
t\right) \mathrm{d}W^{j}\left( t\right) \right)  \notag \\
&&+x^{\phi ^{\ast }}\left( t\right) \dot{\Lambda}\left( t\right) \mathrm{d}%
t+\chi \left( t\right) \mathbb{E}\left[ A\left( t\right) x^{\phi ^{\ast
}}\left( t\right) +B\left( t\right) u^{\phi ^{\ast }}\left( t\right) \right]
\mathrm{d}t+\mathbb{E}\left[ x^{\phi ^{\ast }}\left( t\right) \right] \dot{%
\chi}\left( t\right) \mathrm{d}t+\dot{\varphi}\left( t\right) \mathrm{d}t.
\label{re1}
\end{eqnarray}%
Immediately, we derive that
\begin{equation*}
\sum_{j=1}^{m}q_{j}\left( t\right) \mathrm{d}W^{j}\left( t\right)
=\sum_{j=1}^{m}\Lambda \left( t\right) D_{j}\left( t\right) u^{\phi ^{\ast
}}\left( t\right) \mathrm{d}W^{j}\left( t\right) .
\end{equation*}%
That is
\begin{equation*}
q_{j}\left( t\right) =\Lambda \left( t\right) D_{j}\left( t\right) u^{\phi
^{\ast }}\left( t\right) ,\text{ }j=1,\ldots ,m.
\end{equation*}%
From Theorem \ref{the1}, it yields
\begin{eqnarray*}
0 &=&p\left( t\right) B^{\top }\left( t\right) +\sum_{j=1}^{m}D_{j}^{\top
}\left( t\right) q_{j}\left( t\right) \\
&=&\left( \Lambda \left( t\right) x^{\phi ^{\ast }}(t)+\chi \left( t\right)
\mathbb{E}\left[ x^{\phi ^{\ast }}\left( t\right) \right] +\varphi \left(
t\right) \right) B^{\top }\left( t\right) +\sum_{j=1}^{m}D_{j}^{\top }\left(
t\right) \Lambda \left( t\right) D_{j}\left( t\right) u^{\phi ^{\ast
}}\left( t\right) .
\end{eqnarray*}%
That is%
\begin{equation*}
\Lambda \left( t\right) \sum_{j=1}^{m}D_{j}^{\top }\left( t\right)
D_{j}\left( t\right) u^{\phi ^{\ast }}\left( t\right) =-\left( \Lambda
\left( t\right) x^{\phi ^{\ast }}(t)+\chi \left( t\right) \mathbb{E}\left[
x^{\phi ^{\ast }}\left( t\right) \right] +\varphi \left( t\right) \right)
B^{\top }\left( t\right) .
\end{equation*}%
So%
\begin{eqnarray}
u^{\phi ^{\ast }}\left( t\right) &=&-\Lambda \left( t\right) ^{-1}\Pi
^{-1}\left( t\right) \left( \Lambda \left( t\right) x^{\phi ^{\ast
}}(t)+\chi \left( t\right) \mathbb{E}\left[ x^{\phi ^{\ast }}\left( t\right) %
\right] +\varphi \left( t\right) \right) B^{\top }\left( t\right)  \notag \\
&=&-\Pi ^{-1}\left( t\right) B^{\top }\left( t\right) x^{\phi ^{\ast
}}\left( t\right) -\Lambda \left( t\right) ^{-1}\chi \left( t\right) \Pi
^{-1}\left( t\right) B^{\top }\left( t\right) \mathbb{E}\left[ x^{\phi
^{\ast }}\left( t\right) \right]  \notag \\
&&-\Lambda \left( t\right) ^{-1}\varphi \left( t\right) \Pi ^{-1}\left(
t\right) B^{\top }\left( t\right) .  \label{opcon}
\end{eqnarray}%
For simplicity, we set
\begin{eqnarray*}
\Upsilon _{1}\left( t\right) &=&-\Pi ^{-1}\left( t\right) B^{\top }\left(
t\right) , \\
\Upsilon _{2}\left( t\right) &=&-\Lambda \left( t\right) ^{-1}\chi \left(
t\right) \Pi ^{-1}\left( t\right) B^{\top }\left( t\right) , \\
\Upsilon _{3}\left( t\right) &=&-\Lambda \left( t\right) ^{-1}\varphi \left(
t\right) \Pi ^{-1}\left( t\right) B^{\top }\left( t\right) .
\end{eqnarray*}%
Therefore, we have
%Note that due to $D_{j}\left( \cdot \right) ,Z_{j}^{\Lambda }\left( \cdot
%\right) ,Z_{j}^{\chi }\left( \cdot \right) $, and $Z_{j}^{\varphi }\left(
%\cdot \right) $ are $\mathcal{F}_{t}$-adapted, the processes $\Upsilon
%_{1}\left( t\right) ,\Upsilon _{2}\left( t\right) ,\Upsilon _{3}\left(
%t\right) $, and $\Upsilon _{4}\left( t\right) $ are also $\mathcal{F}_{t}$%
%-adapted. Therefore, we have%
\begin{equation}
\phi ^{\ast }\left( t,x,\mu \right) =\Upsilon _{1}\left( t\right) x+\Upsilon
_{2}\left( t\right) \bar{\mu}+\Upsilon _{3}\left( t\right) .  \label{opcon2}
\end{equation}%
From (\ref{opcon2}), it is easy to compute
\begin{equation}
\partial _{x}\phi ^{\ast }\left( t,x,\mu \right) =\Upsilon _{1}\left(
t\right) ,\text{ }\partial _{\mu }\phi ^{\ast }\left( t,x,\mu ;\hat{x}%
\right) =\Upsilon _{2}\left( t\right) .  \label{opcon3}
\end{equation}%
Plugging (\ref{opcon3}) into (\ref{re1}), we derive that
\begin{eqnarray*}
&&-\Bigg \{\left( A\left( t\right) +B\left( t\right) \Upsilon _{1}\left(
t\right) \right) \left( \Lambda \left( t\right) x^{\phi ^{\ast }}(t)+\chi
\left( t\right) \mathbb{E}\left[ x^{\phi ^{\ast }}\left( t\right) \right]
+\varphi \left( t\right) \right) \\
&&+\mathbb{\hat{E}}\left[ B\left( t\right) \Upsilon _{2}\left( t\right)
\left( \Lambda \left( t\right) \hat{x}^{\phi ^{\ast }}(t)+\chi \left(
t\right) \mathbb{E}\left[ x^{\phi ^{\ast }}\left( t\right) \right] +\varphi
\left( t\right) \right) \right] \\
&&+\sum_{j=1}^{m}D_{j}\left( t\right) \Upsilon _{1}\left( t\right) \Lambda
\left( t\right) D_{j}\left( t\right) \left( \Upsilon _{1}\left( t\right)
x^{\phi ^{\ast }}\left( t\right) +\Upsilon _{2}\left( t\right) \mathbb{E}%
\left[ x^{\phi ^{\ast }}\left( t\right) \right] +\Upsilon _{3}\left(
t\right) \right) \\
&&+\mathbb{\hat{E}}\Bigg [\sum_{j=1}^{m}D_{j}\left( t\right) \Upsilon
_{2}\left( t\right) \Lambda \left( t\right) D_{j}\left( t\right) \left(
\Upsilon _{1}\left( t\right) \hat{x}^{\phi ^{\ast }}\left( t\right)
+\Upsilon _{2}\left( t\right) \mathbb{E}\left[ x^{\phi ^{\ast }}\left(
t\right) \right] +\Upsilon _{3}\left( t\right) \right) \Bigg ]\Bigg \} \\
&=&\Lambda \left( t\right) \left[ A\left( t\right) x^{\phi ^{\ast }}\left(
t\right) +B\left( t\right) u^{\phi ^{\ast }}\left( t\right) \right] \\
&&+\dot{\Lambda}\left( t\right) x^{\phi ^{\ast }}\left( t\right) +\chi
\left( t\right) \mathbb{E}\left[ A\left( t\right) x^{\phi ^{\ast }}\left(
t\right) +B\left( t\right) u^{\phi ^{\ast }}\left( t\right) \right] +\mathbb{%
E}\left[ x^{\phi ^{\ast }}\left( t\right) \right] \dot{\chi}\left( t\right)
\mathrm{d}t+\dot{\varphi}\left( t\right) \\
&=&\Lambda \left( t\right) \left[ A\left( t\right) x^{\phi ^{\ast }}\left(
t\right) +B\left( t\right) \left( \Upsilon _{1}\left( t\right) x^{\phi
^{\ast }}\left( t\right) +\Upsilon _{2}\left( t\right) \mathbb{E}\left[
x^{\phi ^{\ast }}\left( t\right) \right] +\Upsilon _{3}\left( t\right)
\right) \right] \\
&&+\dot{\Lambda}\left( t\right) x^{\phi ^{\ast }}\left( t\right) +\chi
\left( t\right) \mathbb{E}\left[ A\left( t\right) x^{\phi ^{\ast }}\left(
t\right) +B\left( t\right) \left( \Upsilon _{1}\left( t\right) x^{\phi
^{\ast }}\left( t\right) +\Upsilon _{2}\left( t\right) \mathbb{E}\left[
x^{\phi ^{\ast }}\left( t\right) \right] +\Upsilon _{3}\left( t\right)
\right) \right] \\
&&+\mathbb{E}\left[ x^{\phi ^{\ast }}\left( t\right) \right] \dot{\chi}%
\left( t\right) +\dot{\varphi}\left( t\right) ,
\end{eqnarray*}%
from which one can derive the following relations: for $x^{\phi ^{\ast
}}\left( t\right) $%
\begin{eqnarray}
&&-\left[ \left( A\left( t\right) +B\left( t\right) \Upsilon _{1}\left(
t\right) \right) \Lambda \left( t\right) +\sum_{j=1}^{m}D_{j}\left( t\right)
\Upsilon _{1}\left( t\right) \Lambda \left( t\right) D_{j}\left( t\right)
\Upsilon _{1}\left( t\right) \right] x^{\phi ^{\ast }}\left( t\right)  \notag
\\
&=&\left[ \Lambda \left( t\right) \left[ A\left( t\right) +B\left( t\right)
\Upsilon _{1}\left( t\right) \right] +\dot{\Lambda}\left( t\right) \right]
x^{\phi ^{\ast }}\left( t\right) ,  \label{ri1}
\end{eqnarray}%
for $\mathbb{E}\left[ x^{\phi ^{\ast }}\left( t\right) \right] $%
\begin{eqnarray}
&&-\left( A\left( t\right) +B\left( t\right) \Upsilon _{1}\left( t\right)
\right) \chi \left( t\right) \mathbb{E}\left[ x^{\phi ^{\ast }}\left(
t\right) \right] -B\left( t\right) \Upsilon _{2}\left( t\right) \left(
\Lambda \left( t\right) +\chi \left( t\right) \right) \mathbb{E}\left[
x^{\phi ^{\ast }}\left( t\right) \right]  \notag \\
&&-\sum_{j=1}^{m}D_{j}\left( t\right) \Upsilon _{1}\left( t\right) \Lambda
\left( t\right) D_{j}\left( t\right) \Upsilon _{2}\left( t\right) \mathbb{E}%
\left[ x^{\phi ^{\ast }}\left( t\right) \right]  \notag \\
&&-\sum_{j=1}^{m}D_{j}\left( t\right) \Upsilon _{2}\left( t\right) \Lambda
\left( t\right) D_{j}\left( t\right) \left( \Upsilon _{1}\left( t\right)
+\Upsilon _{2}\left( t\right) \right) \mathbb{E}\left[ x^{\phi ^{\ast
}}\left( t\right) \right]  \notag \\
&=&\Lambda \left( t\right) B\left( t\right) \Upsilon _{2}\left( t\right)
\mathbb{E}\left[ x^{\phi ^{\ast }}\left( t\right) \right] +\chi \left(
t\right) \left[ A\left( t\right) +B\left( t\right) \left( \Upsilon
_{1}\left( t\right) +\Upsilon _{2}\left( t\right) \right) \right] \mathbb{E}%
\left[ x^{\phi ^{\ast }}\left( t\right) \right]  \notag \\
&&+\dot{\chi}\left( t\right) \mathbb{E}\left[ x^{\phi ^{\ast }}\left(
t\right) \right] ,  \label{ri2}
\end{eqnarray}%
and
\begin{eqnarray}
&&-\Bigg \{\left( A\left( t\right) +B\left( t\right) \Upsilon _{1}\left(
t\right) \right) \varphi \left( t\right) +B\left( t\right) \Upsilon
_{2}\left( t\right) \varphi \left( t\right)  \notag \\
&&+\sum_{j=1}^{m}D_{j}\left( t\right) \Upsilon _{1}\left( t\right) \Lambda
\left( t\right) D_{j}\left( t\right) \Upsilon _{3}\left( t\right)
+\sum_{j=1}^{m}D_{j}\left( t\right) \Upsilon _{2}\left( t\right) \Lambda
\left( t\right) D_{j}\left( t\right) \Upsilon _{3}\left( t\right) \Bigg \}
\notag \\
&=&\Lambda \left( t\right) B\left( t\right) \Upsilon _{3}\left( t\right)
+\chi \left( t\right) B\left( t\right) \Upsilon _{3}\left( t\right) +\dot{%
\varphi}\left( t\right) .  \label{ri3}
\end{eqnarray}%
Based on (\ref{ri1})-(\ref{ri3}), one can derive (suppressing $t$ in the
functions)%
\begin{equation}
\left\{
\begin{array}{rcl}
\dot{\Lambda}+\left[ 2\left( A-B\Pi ^{-1}B^{\top }\right)
+\sum_{j=1}^{m}\left( D_{j}\Pi ^{-1}B^{\top }\right) ^{2}\right] \Lambda & =
& 0, \\
\Lambda \left( T\right) & = & \lambda ,%
\end{array}%
\right.  \label{ric1}
\end{equation}%
\begin{equation}
\left\{
\begin{array}{rcl}
\dot{\chi}+2\left[ A-2B\Pi ^{-1}B^{\top }+\sum_{j=1}^{m}\left( D_{j}\Pi
^{-1}B^{\top }\right) ^{2}\right] \chi &  &  \\
+\Lambda ^{-1}\left[ \sum_{j=1}^{m}\left( D_{j}\Pi ^{-1}B^{\top }\right)
^{2}-2B\Pi ^{-1}B^{\top }\right] \chi ^{2} & = & 0, \\
\chi \left( T\right) & = & -\lambda ,%
\end{array}%
\right.  \label{ric2}
\end{equation}%
and
\begin{equation}
\left\{
\begin{array}{rcl}
\dot{\varphi}+\Big [A-2B\Pi ^{-1}B^{\top }-2\Lambda ^{-1}\chi B\Pi
^{-1}B^{\top } &  &  \\
+\sum_{j=1}^{m}\left( D_{j}\Pi ^{-1}B^{\top }\right) ^{2}+\Lambda ^{-1}\chi
\sum_{j=1}^{m}\left( D_{j}\Pi ^{-1}B^{\top }\right) ^{2}\Big ]\varphi & = &
0, \\
\varphi \left( T\right) & = & \alpha -1.%
\end{array}%
\right.  \label{ric3}
\end{equation}

%
%
%
%
%
%
%
%
%
%
%
%
%
%
%
%\begin{eqnarray}
%&&-\Bigg \{\left( A-B\Pi ^{-1}B^{\top }\right) \varphi -\chi B\Lambda
%^{-1}\Pi ^{-1}B^{\top }\varphi +\sum_{j=1}^{m}D_{j}\Pi ^{-1}B^{\top
%}D_{j}\Pi ^{-1}B^{\top }\varphi   \notag \\
%&&+\chi \Lambda ^{-1}\sum_{j=1}^{m}D_{j}\Pi ^{-1}B^{\top }D_{j}\Pi
%^{-1}B^{\top }\varphi \Bigg \} \\
%&=&-B\Pi ^{-1}B^{\top }\varphi -\chi \Lambda ^{-1}B\Pi ^{-1}B^{\top }\varphi
%+\dot{\varphi}.
%\end{eqnarray}%
%
%\begin{eqnarray}
%&&-\left( A-B\Pi ^{-1}B^{\top }\right) \chi +B\Lambda ^{-1}\chi \Pi
%^{-1}B^{\top }\left( \Lambda +\chi \right)   \notag \\
%&&-\sum_{j=1}^{m}D_{j}\Pi ^{-1}B^{\top }D_{j}\Pi ^{-1}B^{\top }\chi   \notag
%\\
%&&-\sum_{j=1}^{m}D_{j}\Pi ^{-1}B^{\top }D_{j}\left( \Pi ^{-1}B^{\top
%}+\Lambda ^{-1}\chi \Pi ^{-1}B^{\top }\right) \chi   \notag \\
%&=&-B\Pi ^{-1}B^{\top }\chi +\chi \left[ A-B\left( \Pi ^{-1}B^{\top
%}+\Lambda ^{-1}\Pi ^{-1}B^{\top }\chi \right) \right] +\dot{\chi},
%\end{eqnarray}%
%\begin{eqnarray}
%&&-2\sum_{j=1}^{m}D_{j}\Pi ^{-1}B^{\top }D_{j}\Pi ^{-1}B^{\top }\chi  \\
%&&-\sum_{j=1}^{m}D_{j}\Pi ^{-1}B^{\top }D_{j}\Lambda ^{-1}\Pi ^{-1}B^{\top
%}\chi ^{2}  \notag \\
%&=&2\left( A-B\Pi ^{-1}B^{\top }\right) \chi -2B\Pi ^{-1}B^{\top }\chi
%-2\Lambda ^{-1}B\Pi ^{-1}B^{\top }\chi ^{2}+\dot{\chi},
%\end{eqnarray}

\section{Concluding remarks}

\label{sect4}

This paper rigorously addresses the continuous-time mean--variance portfolio
selection problem by unifying BSDE theory with Kantorovich duality. We
develop a novel analytical framework tailored to feedback control design for
McKean--Vlasov stochastic systems. Through the joint resolution of three
coupled systems of ODEs, we obtain an explicit, closed-loop optimal
investment policy characterized entirely in terms of observable state
variables. Hopefully, deep learning-based numerical methods are a possible
direction for future work.

\section*{Conflict of Interest}

The authors declare no conflict of interest.

\end{document}